\documentclass{amsart}
\usepackage{inputenc, amsthm, amssymb, amsmath}

\newtheorem{theorem}{Theorem}
\newtheorem{definition}{Definition}
\newtheorem{proposition}{Proposition}
\newtheorem{corollary}{Corollary}

\DeclareMathOperator{\Norm}{Norm}

\newcommand {\frakp}{\mathfrak{p}}

\newcommand {\calD} {\mathcal {D}}
\newcommand {\calO} {\mathcal {O}}

\newcommand {\SL}{\mathrm{SL}}
\newcommand {\bbN} {\mathbb {N}}
\newcommand {\bbQ} {\mathbb {Q}}
\newcommand {\bbR} {\mathbb {R}}
\newcommand {\bbC} {\mathbb {C}}

\newcommand{\ord}{\mathop{\mathrm{ord}}}
\newcommand{\Res}{\mathop{\mathrm{Res}}}
\renewcommand {\Re}{\mathop{\mathrm{Re}}}
\renewcommand {\Im}{\mathop{\mathrm{Im}}}

\title{Asymptotic properties of Dedekind zeta functions in families of number fields}
\author{Alexey Zykin}
\address{
State University --- Higher School of Economics, 
117312, Vavilova st., 7, Moscow, Russia
\newline \indent
Laboratoire Poncelet (UMI 2615)
}
\email{alzykin@gmail.com}
\date{}

\begin{document}
\begin{abstract}
The main goal of this paper is to prove a formula that expresses the limit behaviour of Dedekind zeta functions for $\Re s > 1/2$ in families of number fields, assuming that the Generalized Riemann Hypothesis holds. This result can be viewed as a generalization of the Brauer--Siegel theorem. As an application we obtain a limit formula for Euler--Kronecker constants in families of number fields.
\end{abstract}
\maketitle 

\section{Introduction}
Our starting point is the classical Brauer--Siegel theorem for number fields first proven by Siegel in the case of quadratic fields and then by Brauer (see \cite{Bra}) in a more general situation. This theorem states that if $K$ runs through a sequence of number fields normal over $\mathbb{Q}$ such that $n_K/\log|D_K|\to 0,$ then $\log (h_K R_K)/\log \sqrt{|D_K|}\to 1.$ Here $D_K,$ $h_K,$ $R_K$ and $n_K$ are respectively the discriminant, the class number, the regulator and the degree of the field $K.$

In \cite{Tsfa} this theorem was generalized by Tsfasman and Vl\u{a}du\c{t} to the case when the condition $n_K/\log|D_K|\to 0$ no longer holds. To formulate this result we will need to introduce some notation. 

For a finite extension $K/\bbQ,$ let $\Phi_q(K)$ be the number of prime ideals of the ring of integers $\calO_K$ with norm $q$, i.e. $\Phi_q(K)=|\{\frakp\mid \Norm \frakp =q\}|.$ Furthermore, denote by $\Phi_{\bbR}(K)$ and $\Phi_{\bbC}(K)$ the number of real and complex places of $K$ respectively. Let $g_K=\log\sqrt{|D_K|}$ be the genus of the field $K$ (in analogy with the function field case). An extension $K/\bbQ$ is called almost normal if there exists a tower of extensions $K=K_n\supseteq K_{n-1} \dots \supseteq K_1 \supseteq K_0=\bbQ$ such that $K_i/K_{i-1}$ is normal for all $i.$ 

Consider a family of pairwise non-isomorphic number fields $\{K_i\}.$

\begin{definition}
If the limits $\phi_{\alpha}=\lim\limits_{i\to\infty}\frac{\Phi_{\alpha}(K_i)}{g_{K_i}}, \alpha\in \{\bbR, \bbC, 2, 3, 4, 5, 7, 9, \dots \}$ exist for each $\alpha$ then the family $\{K_i\}$ is called asymptotically exact. It is asymptotically good if there exists $\phi_\alpha\neq 0$ and asymptotically bad otherwise. The numbers $\phi_{\alpha}$ are called the Tsfasman--Vl\u{a}du\c{t} invariants of the family $\{K_i\}.$ 
\end{definition}

It is not difficult to check (see \cite[Lemma 2.7]{Tsfa}) that the condition $n_K/\log|D_K|\to 0$ from the Brauer--Siegel theorem is equivalent to the fact that the corresponding family is asymptotically bad. One can prove that any family contains an asymptotically exact subfamily and that an infinite tower of number fields is always asymptotically exact (see \cite[Lemma 2.2 and Lemma 2.4]{Tsfa}).

Now we can formulate the Tsfasman--Vl\u{a}du\c{t} theorem proven in \cite[Theorem 7.3]{Tsfa} in the asymptotically good case and in \cite[Theorem 1]{Zy} in the asymptotically bad one.

\begin{theorem}
\label{assgood}
For an asymptotically exact family $\{K_i\}$ we have
\begin{equation}
\label{TVZ}
\lim\limits_{i\to\infty}\frac{\log (h_{K_i}R_{K_i})}{g_{K_i}}=1+\sum_{q}\phi_q \log \frac{q}{q-1} - \phi_\mathbb{R} \log 2 - \phi_\mathbb{C}\log 2\pi,
\end{equation}
provided either all $K_i$ are almost normal over $\bbQ$ or the Generalized Riemann Hypothesis (GRH) holds for zeta functions of the fields $K_i$.
\end{theorem}

To generalize this theorem still further we will have to use the concept of limit zeta functions from \cite{Tsfa}.

\begin{definition}
The limit zeta function of an asymptotically exact family of number fields $\{K_i\}$ is defined as
$$\zeta_{\{K_i\}}(s)=\prod\limits_{q}(1-q^{-s})^{-\phi_q}.$$
\end{definition}

Theorem C from \cite{Tsfa} gives us the convergence of the above infinite product for $\Re s\geq 1$. Let $\varkappa_K=\Res\limits_{s=1} \zeta_K(s)$ be the residue of the Dedekind zeta function of the field $K$ at $s=1.$ Using the residue formula (see \cite[Chapter VIII, Theorem 5]{Lan})
$$
\varkappa_K=\frac{2^{\Phi_{\bbR}(K)}(2\pi)^{\Phi_{\bbC}(K)}h_K R_K}{w_K\sqrt{|D_K|}}
$$
(here $w_K$ is the number of roots of unity in $K$) and the estimate $w_K =O(n_K^2)$ (see \cite[p. 322]{Lan}) one can see that the question about the behaviour of the ratio from the Brauer--Siegel theorem is  immediately reduced to the corresponding question for $\varkappa_K.$ 

The formula (\ref{TVZ}) can be rewritten as $\lim\limits_{i\to\infty}\frac{\log \varkappa_{K_i}}{g_{K_i}}=\log\zeta_{\{K_i\}}(1).$ Furthermore, Tsfasman and Vl\u{a}du\c{t} prove in \cite[Proposition 4.2]{Tsfa} that for $\Re s>1$ the equality $\lim\limits_{i\to\infty}\frac{\log \zeta_{K_i}(s)}{g_{K_i}}=\log\zeta_{\{K_i\}}(s)$ holds.

Our main goal is to investigate the question of the validity of the above equality for $\Re s<1.$ We work in the number field case, for the function field case see \cite{Zynew}, where the same problem was treated in a much broader context.

The case $s=1$ is in a sense equivalent to the Brauer--Siegel theorem so current techniques does not allow to treat it in full generality without the assumption of GRH. From now on we will assume that GRH holds for Dedekind zeta functions of the fields under consideration. Assuming GRH, Tsfasman and Vl\u{a}du\c{t} proved (\cite[Corollary from Theorem A]{Tsfa}) that the infinite product for $\zeta_{\{K_i\}}(s)$ is absolutely convergent for $\Re s \geq \frac{1}{2}.$ We can now formulate our main results.

\begin{theorem}
\label{main1}
Assuming GRH, for an asymptotically exact family of number fields $\{K_i\}$ for $\Re s > \frac{1}{2}$ we have 
$$\lim\limits_{i\to\infty}\frac{\log ((s-1)\zeta_{K_i}(s))}{g_{K_i}}=\log\zeta_{\{K_i\}}(s).$$
The convergence is uniform on compact subsets of the half-plane $\{s \mid \Re s > \frac{1}{2}\}.$
\end{theorem}

The result for $s = \frac{1}{2}$ is considerably weaker and we can only prove the following upper bound:
\begin{theorem}
\label{main2}
Let $\rho_{K_i}$ be the first non-zero coefficient in the Taylor series expansion of $\zeta_{K_i}(s)$ at $s=\frac{1}{2},$ i.~e. $$\zeta_{K_i}(s)=\rho_{K_i}\left(s-\frac{1}{2}\right)^{r_{K_i}}+o\left(\left(s-\frac{1}{2}\right)^{r_{K_i}}\right).$$
Then, assuming GRH, for any asymptotically exact family of number fields $\{K_i\}$ the following inequality holds:
\begin{equation}
\label{ineq}
\limsup\limits_{i\to\infty}\frac{\log |\rho_{K_i}|}{g_{K_i}}\leq\log\zeta_{\{K_i\}}\left(\frac{1}{2}\right).
\end{equation}
\end{theorem}

The question whether the equality holds in theorem \ref{main2} is rather delicate. It is related to the so called low-lying zeroes of zeta functions, that is the zeroes of $\zeta_K(s)$ having small imaginary part compared to $g_K.$ We doubt that the equality $\lim\limits_{i\to\infty}\frac{\log |\rho_{K_i}|}{g_{K_i}}=\log\zeta_{\{K_i\}}(\frac{1}{2})$ holds for any asymptotically exact family $\{K_i\}$ since the behaviour of low-lying zeroes is known to be rather random. Nevertheless, it might hold for "most" families (whatever it might mean). A more thorough discussion of this question in a slightly different situation (low-lying zeroes of $L$-functions of modular forms on $\SL_2(\bbR)$) can be found in \cite{Iwaniec}. 

To illustrate how hard the question may be, let us quote the following result by Iwaniec and Sarnak, which is the object of the paper \cite{IS}. They manage to prove that there exists a sequence $\{d_i\}$ in $\bbN$ of density at least $\frac{1}{2}$  such that
$$\lim\limits_{i\to\infty}\frac{\log \zeta_{\bbQ(\sqrt{d_i})}(\frac{1}{2})}{\log d_i}=\log\zeta_{\{\bbQ(\sqrt{d_i})\}}\left(\frac{1}{2}\right)=0.$$
The techniques of the evaluation of mollified moments of Dirichlet $L$-functions used in that paper are rather involved. We also note that, to our knowledge, there has been no investigation of low-lying zeroes of $L$-functions of growing degree. It seems that the analogous problem in the function field has neither been very well studied.

Let us formulate a corollary of the theorem \ref{main1}. We will need the following definition:

\begin{definition}
The Euler--Kronecker constant of a number field $K$ is defined as $\gamma_K=\frac{c_{0}(K)}{c_{-1}(K)},$ where $\zeta_K(s)=c_{-1}(K)(s-1)^{-1}+c_0(K)+O(s-1).$ 
\end{definition}

Ihara made an extensive study of the Euler-Kronecker constant in \cite{Ih}. In particular, he obtained an asymptotic formula for the behaviour of $\gamma$ in families of curves over finite fields. As a corollary of theorem \ref{main1}, we prove the following analogue of Ihara's result in the number field case:

\begin{corollary}
\label{EK}
Assuming GRH, for any asymptotically exact family of number fields $\{K_i\}$ we have $$\lim\limits_{i\to\infty}\frac{\gamma_{K_i}}{g_{K_i}}=-\sum\limits_q\phi_q \frac{\log q}{q-1}.$$
\end{corollary}

This result was formulated in \cite{Tsfa1} without the assumption of the Riemann hypothesis. Unfortunately, the proof given there is flawed. It uses an unjustified change of limits in the summation over prime powers and the limit taken over the family $\{K_i\}.$ Thus, the question about the validity of this equality without the assumption of GRH is still open. It would be interesting to have a result of this type at least under a certain normality condition on our family $\{K_i\}.$ Even the study of abelian extensions is not uninteresting in this setting.

\section{Proofs of the main results}

\begin{proof}[Proof of theorem \ref{main1}]

The statement of the theorem is known for $\Re s > 1$ (see \cite[Proposition 4.2]{Tsfa}) thus we can freely assume that $\Re s < 2.$

We will use the following well known result \cite[Proposition 5.7]{IK} which can be proven using Hadamard's factorization theorem.
\begin{proposition}
\begin{enumerate}
\item For $-\frac{1}{2}\leq \sigma \leq 2, s = \sigma+i t$ we have
$$\frac{\zeta_K'(s)}{\zeta_K(s)}+\frac{1}{s}+\frac{1}{s-1}-\sum_{|s-\rho|<1} \frac{1}{s-\rho} = O(g_K),$$
where $\rho$ runs through all non-trivial zeroes of $\zeta_K(s)$ and the constant in $O$ is absolute.
\item The number $m(T, K)$ of zeroes $\rho=\beta+\gamma i$ of $\zeta_K(s)$ such that $|\gamma-T|\leq 1$ satisfies $m(T, K)< C (g_K+n_K\log(|T|+4))$ with an absolute constant $C.$
\end{enumerate}
\end{proposition}

Now, applying this proposition, we see that for fixed $T > 0, \epsilon > 0$ and any $s\in \calD_{T,\epsilon}=\{s\in \bbC\mid |\Im s| \leq T, \epsilon + \frac{1}{2} \leq \Re s \leq 2 \}$ we have
\begin{equation}
\label{eq1}
\frac{\zeta_K'(s)}{\zeta_K(s)} + \frac{1}{s-1} = \sum_{|s - \rho| < \epsilon} \frac{1}{s-\rho} +O_{T, \epsilon}(g_K),
\end{equation}
for by Minkowski's theorem \cite[Chapter V, Theorem 4]{Lan} $n_K < C g_K$ with an absolute constant $C.$

If we assume GRH, the sum over zeroes on the right hand side of (\ref{eq1}) disappears. Integrating, we finally get that in $\calD_{T,\epsilon}$ 
$$\frac{\log (\zeta_K(s)(s-1))}{g_K}=O_{T, \epsilon}(1)$$

Now, we can use the so called Vitali theorem \cite[5.21]{Tit}:
\begin{proposition}
Let $f_n(s)$ be a sequence of functions holomorphic in a domain $\calD.$ Assume that for some $M\in \bbR$ we have $|f_n(s)|<M$ for any $n$ and $s\in\calD.$ Let also $f_n(s)$ tend to a limit at a set of points having a limit point in $\calD.$ Then the sequence $f_n(s)$ tends to a holomorphic function in $\calD$ uniformly on any closed disk contained in $\calD.$
\end{proposition}

It suffices to notice that the convergence of $\log \zeta_{K_i}(s)/ g_{K_i}$ to $\zeta_{\{K_i\}}(s)$ is known for $\Re s > 1$ by \cite[proposition 4.2]{Tsfa}. So, applying the above theorem and using the fact that under GRH $\zeta_{\{K_i\}}(s)$ is holomorphic for $\Re s \geq \frac{1}{2}$ \cite[corollary from theorem A]{Tsfa} we get the required result.
\end{proof}

\begin{proof}[Proof of theorem \ref{main2}]
Denote $g_k=g_{K_k}.$ Let us write 
$$\zeta_{K_k}(s)=c_k \left(s-\frac{1}{2}\right)^{r_k} F_k(s),$$
where $F_k(s)$ is an analytic function in the neighbourhood of $s=\frac{1}{2}$ such that $F_k\left(\frac{1}{2}\right)=1.$
Let us put $s=\frac{1}{2}+\theta,$ where $\theta>0$ is a small positive real number. We have
\begin{equation*}
\frac{\log \zeta_{K_k}(\frac{1}{2}+\theta)}{g_k}=\frac{\log c_k}{g_k}+r_k\frac{\log \theta}{g_k}+\frac{\log F_k(\frac{1}{2}+\theta)}{g_k}.
\end{equation*}
To prove the theorem we will construct a sequence $\theta_k$ such that
\begin{enumerate}
\item $\frac{1}{g_k}\log \zeta_{K_k}\left(\frac{1}{2}+\theta_k\right) \to \log \zeta_{\{K_k\}}\left(\frac{1}{2}\right);$
\item $\frac{r_k}{g_k}\,\log \theta_k \to 0;$
\item $\liminf \frac{1}{g_k}\log F_k\left(\frac{1}{2}+\theta_k\right)\geq 0.$
\end{enumerate}

For each natural number $N$ we choose $\theta(N)$ a decreasing sequence such that 
\begin{equation*}
\left|\zeta_{\{K_k\}}\left(\frac{1}{2}\right)-\zeta_{\{K_k\}}\left(\frac{1}{2}+\theta(N)\right)\right| < \frac{1}{2N}.
\end{equation*}
This is possible since $\zeta_{\{K_k\}}(s)$ is continuous for $\Re s \geq \frac{1}{2}$ by \cite[corollary from theorem A]{Tsfa}. Next, we choose a sequence $k'(N)$ with the property:
\begin{equation*}
\left|\frac{1}{g_{k}}\log \zeta_{K_k}\left(\frac{1}{2}+\theta \right) - \log \zeta_{\{K_k\}}\left(\frac{1}{2}+\theta\right)\right|< \frac{1}{2N}
\end{equation*}
for any $\theta\in[\theta(N+1), \theta(N)]$ and any $k\geq k'(N).$ This is possible by theorem \ref{main1}. Then we choose $k''(N)$ such that 
\begin{equation*}
\frac{-r_k\log \theta(N+1)}{g_k}\leq \frac{\theta(N)}{N}
\end{equation*}
for any $k\geq k''(N),$ which can be done thanks to the following proposition (c.f. \cite[Proposition 5.34]{IK}):
\begin{proposition}
Assume that GRH holds for $\zeta_K(s).$ Then
$$\ord_{s=\frac{1}{2}}\zeta_K(s)< \frac{C \log 3|D_K|}{\log\log 3|D_K|}$$
the constant $C$ being absolute.
\end{proposition} 

Finally, we choose an increasing sequence $k(N)$ such that $k(N)\geq \max(k'(N), k''(N))$ for any $N.$

Now, if we define $N=N(k)$ by the inequality $k(N)\leq k \leq k(N+1)$ and let $\theta_k=\theta(N(k)),$ then from the conditions imposed on $\theta_k$ we automatically get (1) and (2). The delicate point is (3). We will use Hadamard's product formula \cite[p. 137]{Sta}:
\begin{multline*}
\log{\vert D_K\vert}=\Phi_{\bbR}(K)(\log\pi-\psi(s/2))+2\Phi_{\bbC}(K)(\log(2\pi)-\psi(s))-\frac{2}{s}-\frac{2}{s-1}\\
+2{\sum_{\rho}}'\frac{1}{s-\rho}-2\,\frac{\zeta'_K(s)}{\zeta_K(s)},
\end{multline*}
where $\psi(s)=\Gamma'(s)/\Gamma(s)$ is the logarithmic derivative of the gamma function. In the first sum $\rho$ runs over the zeroes of $\zeta_K(s)$ in the critical strip and $\sum'$ means that $\rho$ and $\bar{\rho}$ are to be grouped together. This can be rewritten as
\begin{multline*}
\frac{1}{g_k}\left(\log \zeta_k\left(\frac{1}{2}+\theta\right)-r_k\log \theta\right)'=-1 +\frac{\Phi_{\bbR}(K_k)}{2g_k}\left(\log\pi-\psi\left(\frac{1}{4}+\frac{\theta}{2}\right)\right)+\\ 
+\frac{\Phi_{\bbC}(K_k)}{g_k}\left(\log 2\pi-\psi\left(\frac{1}{2}+\theta\right)\right) + \frac{8\theta}{(4\theta^2-1)g_k}+{\sum_{\rho\neq 1/2}}'\frac{1}{(1/2+\theta-\rho)g_k}.
\end{multline*}
One notices that all the terms on the right hand side except for $-1$ and $\frac{16\theta}{(4\theta^2-1)g_k}$ are positive. Thus, we see that $\frac{1}{g_k}\left(\log F_k\left(\frac{1}{2}+\theta\right)\right)'\geq C$ for any small enough $\theta,$ where $C$ is an absolute constant. From this and from the fact that $F_k\left(\frac{1}{2}\right)=1$ we deduce that
$$\frac{1}{g_k}\log F_k\left(\frac{1}{2}+\theta_k\right)\geq C\theta_k\to 0.$$ 
This proves (3) as well as the theorem.
\end{proof}

\begin{proof}[Proof of the corollary \ref{EK}]
It suffices to take the values at $s=1$ of the derivatives of both sides of the equality in theorem \ref{main1}. This is possible since the convergence is uniform for $\Re s > \frac{1}{2}.$
\end{proof}

\begin{flushleft}
{\bf Acknowledgements.} I would like to thank my advisor Michael Tsfasman for many fruitful discussions and constant attention to my work. I would also like to thank Michel Balazard for giving valuable advices and sharing his ideas with me.
\end{flushleft}

\end{document}